\documentclass[conference]{IEEEtran}
\ifCLASSINFOpdf
\else
\fi
\usepackage{amsthm}
\usepackage{amsfonts}
\usepackage{amsmath}
\usepackage{amssymb}
\usepackage{verbatim}
\usepackage{latexsym}

\theoremstyle{definition}
\newtheorem{defi}{Definition}

 \theoremstyle{plain}
\newtheorem{thm}[defi]{Theorem}
\newtheorem{lem}[defi]{Lemma}
\newtheorem{cor}[defi]{Corollary}

\newtheorem{qn}[defi]{Question}

\begin{document}
%
\title{Planar Difference Functions}



%
\author{\IEEEauthorblockN{Joanne L. Hall\IEEEauthorrefmark{1},
Asha Rao\IEEEauthorrefmark{2}, and
Diane Donovan\IEEEauthorrefmark{3},}
\IEEEauthorblockA{\IEEEauthorrefmark{1}Department of Algebra, Charles University in Prague\\
186 75 Praha 8, Sokolovska 83, Czech Republic \\
}
\IEEEauthorblockA{\IEEEauthorrefmark{2}\IEEEauthorrefmark{1} School of Mathematical and Geospatial Sciences, RMIT University\\
GPO Box 2476
Melbourne 3001, Australia\\
}
\IEEEauthorblockA{\IEEEauthorrefmark{3}School of Mathematics and Physics, University of Queensland\\
St Lucia, Brisbane 4072, Australia\\
}
\IEEEauthorblockA{\IEEEauthorrefmark{1}Email: hall@karlin.mff.cuni.cz\\
\IEEEauthorrefmark{2}Email: asha@rmit.edu.au\\
\IEEEauthorrefmark{3}Email:dmd@maths.uq.edu.au}
}



\maketitle

\begin{abstract}
In 1980 Alltop produced a family of cubic phase sequences that nearly meet the Welch bound for maximum non-peak correlation magnitude. This family of sequences were shown by Wooters and Fields to be useful for quantum state tomography. Alltop's construction used a function that is not planar, but whose difference function is planar. In this paper we show that Alltop type functions cannot exist in fields of characteristic 3 and that for a known class of planar functions, $x^3$ is the only Alltop type function.

\end{abstract}


%
\IEEEpeerreviewmaketitle

\section{Introduction}

Planar functions belong to the larger class of highly nonlinear functions which are of use in classical cryptographic systems, coding theory as well as being of theoretical interest \cite{CD04}.

Let $\mathbb{F}_{p^r}$ be a field of characteristic $p$. A function $f : \mathbb{F}_{p^r} \rightarrow \mathbb{F}_{p^r}$ is  called a \emph{planar
function} if for every  $a \in \mathbb{F}_{p^r}^*$ the function
$\Delta_{fa}: x \mapsto f(a + x)- f (x)$ is a bijection.

An equivalent definition of a planar function involves the ability to construct an affine plane \cite[\S 5]{dem97}, which is where the name planar originates.  It is known that planar functions do not exist on fields of characteristic $2$ \cite[\S 5]{dem97}.

Two orthonormal bases $B_1$ and $B_2$ of  $\mathbb{C}^d$ are \emph{unbiased} if $|\langle \vec{x}|\vec{y}\rangle|=\frac{1}{\sqrt{d}}$ for all $\vec{x}\in B_1$ and $\vec{y}\in B_2$. A set of bases for $\mathbb{C}^d$ which are pairwise unbiased is a set of \emph{mutually unbiased bases} (MUBs).
This idea is credited to Schwinger \cite{Schw60} who realised that
a quantum system prepared in a basis state from $B'$ reveals no
information when measured with respect to the basis $B$. Mutually unbiased bases (MUBs) are an important tool in quantum information theory and have applications in quantum cryptography \cite{BB84,SARG04} and  quantum state tomography \cite{WF89}.


We highlight  two different constructions of mutually unbiased bases for
odd prime power integers, one which uses a polynomial of degree 3
but only works for finite fields of characteristic $p\geq 5$
\cite[Theorem 1]{KR03}. This construction is a generalisation
of low correlation complex sequences first constructed by Alltop
\cite{Alltop80} for spread spectrum radar and communication
applications.

The other construction uses a polynomial of degree
2, which is a planar function, to construct the vectors in a complete set of MUBs (see Thm \ref{royscott}) \cite{RS07,KR03}. In contrast, the Alltop construction of complete sets of MUBs (Thm \ref{alltop} \cite{Alltop80, KR03}) uses a function, $f(x)=x^3$  which is not planar, but for which the
difference function $\Delta_{fa}(x)$ is planar.  We aim to discover if $x^3$ is the only polynomial function
with this property.

Let $\omega_p:=e^{\frac{2i\pi}{p}}$.

\begin{thm}[Planar function construction]\label{royscott}\cite[Thm 4]{DY2007}\cite[Thm 4.1]{RS07}
Let $\mathbb{F}_q$ be a field of odd characteristic $p$.  Let $\Pi(x)$ be a planar function on $\mathbb{F}_q$.  Let $V_a:=\{\vec{v}_{ab}:b\in \mathbb{F}_q\}$ be the set of vectors
\begin{equation}\label{planarv}
\vec{v}_{ab}= \frac{1}{\sqrt{q}}\left(\omega_p^{\mbox{\emph{tr}}(a\Pi(x)+bx)}\right)_{x\in\mathbb{F}_q}\
\end{equation}
with $a,b\in\mathbb{F}_q$.  The standard basis $E$ along with the sets $V_a$, $a\in \mathbb{F}_q$, form a complete set of $q+1$ MUBs  in $\mathbb{C}^q$.
\end{thm}
The property that is being exploited in the planar function construction is that \cite[Theorem 16]{CD04}
\begin{equation}\label{eqn:planarsum}
\left|\sum_{x\in\mathbb{F}_q}\omega_p^{\textrm{tr}(\Pi(x))}\right|=\sqrt{q}.
\end{equation}

\begin{thm}[Alltop Construction]\cite[Thm 1]{KR03}\label{alltop}
Let $\mathbb{F}_q$ be a finite field of odd characteristic $p\geq5$ and let $V_a:=\{\vec{v}_{ab}:b\in \mathbb{F}_q\}$ be the set of vectors
\begin{equation}\label{Alltopv}
\vec{v}_{ab}:= \frac{1}{\sqrt{q}}\left(\omega_p^{\mbox{\emph{tr}}((x+a)^3+b(x+a))}\right)_{x\in\mathbb{F}_q}
\end{equation}
with $a,b\in\mathbb{F}_q$. The standard basis $E$ along with the sets $V_a$, $a\in \mathbb{F}_q$, form a complete set of $q+1$ MUBs  in $\mathbb{C}^q$.
 \end{thm}
Although on the surface the Alltop construction does not use a planar function, when inspecting the inner product of the vectors,
\begin{align}
\langle\vec{v}_{ab}|\vec{v}_{cd}\rangle= &  \frac{1}{q}\bigg|\sum_{x\in\mathbb{F}_q}\omega_p^{\textrm{tr}[3(a-c)x^2 + (3a^2-3c^2 + b-d)x} \nonumber\\
& \quad\quad\quad\quad{}^{+ (a^3-c^3 + ba -dc)]}\bigg|
\end{align}
we notice a polynomial of degree 2. Now 
$\Pi(x)=x^2$ is a planar function, and equation (\ref{eqn:planarsum}) ensures that a set of MUBs has been constructed.  Thus if we take $f(x)=(x+b)^3$ then $\Delta_{fd}(x)$ is a planar function. The question that arises is whether $f(x)=(x+b)^3$ is the only function of this type:
\begin{qn}\label{qn:1}
Is $f(x)=x^3$ the only polynomial function on a Galois field such
that $\Delta_{fa}(x)$ is a planar function?
\end{qn}

Two sets of MUBs, $\mathcal{B}=\{B_0,B_1,\dots,B_d\}$ and $\mathcal{C}=\{C_0,C_1,\dots,C_d\}$, written as matrices, are \emph{equivalent} \cite[App A]{BWB2010} if either $\mathcal{B}$ or $\mathcal{B}^*$ is equal to $ \{UC_0D_0P_0, UC_1D_1P_1,\dots ,UC_dD_dP_d\}$ for some  unitary matrix, $U$, 
 unitary diagonal matrices, $D_i$,  and permutation matrices, $P_i$.

Godsil and Roy  \cite{GR09} have shown that the Alltop construction produces MUBs that are equivalent to the set of MUBs constructed using $\Pi(x)=x^2$ in the  Planar function construction, which naturally leads to the following question:
\begin{qn}
If another function exists such that $\Delta_{fa}(x)$ is planar, will the sets of MUBs constructed be equivalent?
\end{qn}

Any function which meets  the criteria of Question
\ref{qn:1} will hence forth be called an \emph{Alltop} type
function.
\begin{defi}An \emph{Alltop} type polynomial is a polynomial, $f$, such that for each $a\in\mathbb{F}_q^*$
\begin{equation}
\Delta_{fa}(x)=\Pi_a(x)
\end{equation}
 for some planar polynomial $\Pi_a$.
\end{defi}

We investigate Question \ref{qn:1} establishing that Alltop type functions cannot exist on fields of characteristic 3, and show that for the class of planar functions of the form $\Pi(x)=x^{p^k+1}$ with $p$ a prime, $f(x)=x^3$ is the only Alltop type function.

\section{Preliminaries}
We begin with some preliminary  results concerning
polynomials.    The following properties of binomial expansions will be used in calculating $\Delta_{fa}(x)$.

\begin{lem}\cite[Prop. 8]{Sing80}\label{nondivbinom}
Let $n = \sum a_ip^i$  and $k = \sum b_ip^i$  with $0\leq
a_i,b_i \leq p$. Then  $p \nmid \binom{n}{k}$ if and only
if $0 \leq b_i \leq a_i$ for all $i$.
\end{lem}

\begin{lem}\cite[Cor 19.1]{Sing80}\label{primedivBinom}
If $p^s | n$ and $(k,p) = 1$, then $p^s | \binom{n}{k}$.
\end{lem}

\begin{lem}\cite[Cor 10.2] {Sing80}\label{binomprime}
\begin{equation}\binom{p^s}{k} = \left \{ \begin{array}{l l}
 1 \pmod{p} & \mathrm{if}\; k \in \{0, p^s\}\\
 0 \pmod{p} & \mathrm{if}\; 1 \leq k \leq p^s - 1
 \end{array} \right .
\end{equation}
\end{lem}
\begin{cor}\label{lem}
\begin{equation}\binom{p^s+1}{k}  \begin{array}{l l}
 \neq 0 \pmod{p} & \mathrm{if}\; k \in \{ 0,1,p^s,p^s+1\} \\
 =0 \pmod{p} & \mathrm{if}\; 2 \leq k \leq p^s - 1
 \end{array}
\end{equation}
\begin{equation}\binom{p^s+2}{k}  \begin{array}{l l}
 \neq 0 \pmod{p} & \mathrm{if}\; k \in \{ 0,1,2,p^s,\\
 & \quad\quad\quad p^s+1,p^s+2\} \\
 =0 \pmod{p} & \mathrm{if}\; 3 \leq k \leq p^s - 1
 \end{array}
\end{equation}
\end{cor}

\begin{lem}\label{binomRed}
$k \binom{n}{k} = n \binom{n-1}{k-1}$.
\end{lem}

Using these preliminary facts, we can calculate a few properties of $\Delta_{fa}(x)$ when $f$ is a monomial.

\begin{lem}\label{constant}
 Let
 $\mathbb{F}_q$ be a field of characteristic $p$ and $f(x)=x^{n}$ with 
 $n = p^s$, where $s \geq 0$. Then $\Delta_{fa}(x)$
  is constant for all $a\in\mathbb{F}_q^*$.
\end{lem}
\begin{proof}By the Taylor's expansion
$\Delta_{fa}(x) = \sum_{i = 1}^n \binom{n}{i}a^i x^{n-i}$.
By  Lemma \ref{binomprime} $p | \binom{n}{k} $ for all
$k\in\{1,\dots,p^s-1\}$ hence $(x + a )^n - x^n$ is
constant.
\end{proof}

\begin{thm}\label{thm:m-1}
Let $\mathbb{F}_q$ be a field of  characteristic $p$ and $f(x)=x^n$ with $n =
p^s m$ where $s \geq 0, m > 1$ and $(m,p) = 1$. Then $\Delta_{fa}(x)$ has degree $p^s (m - 1)$ for all $a\in\mathbb{F}_q^*$.
\end{thm}
\begin{proof} Let $f(x) = x^n$  then by the Taylor's expansion
$\Delta_{fa}(x) = \sum_{i = 1}^n \binom{n}{i}a^i x^{n-i}$.
We need to show that the first  non-zero coefficient in
this binomial expansion is $\binom{n}{p^s}$.

We start with $s = 0$. Then $n$ and $p$ are co-prime and
$\binom{n}{p^s} = \binom{n}{1} = n$ which is not divisible
by $p$.

Next consider $s > 0$ and recall $m > 1$.

If $s = 1$, then $n = pm$ and, by Lemma \ref{primedivBinom},  $p | \binom{n}{i} $  for $1
\leq i \leq p -1$ but $ p \nmid \binom{n}{p}$ which is the coefficient of $x^{n-p} = x^{p(m-1)}$ and hence
$\Delta_{fa}(x)$ has degree $p(m-1)$.

For $s > 1$, it is clear from Lemma \ref{primedivBinom}
that $p | \binom{n}{i} $ for all $i$ such that $(p,i) = 1$.
The question that remains is whether $p | \binom{n}{i} $
for those $i < p^s$ with $i,p$ not co-prime.

Let $i = p^k t$ for $1 \leq k < s$ and $(p,t) = 1, t \geq 1$.

By Lemmas \ref{primedivBinom} and \ref{binomRed},
\begin{align}
\binom{n}{p^kt} & =\frac{n}{p^kt} \binom{n - 1}{p^kt - 1} \\
             & =\frac{p^{s-k}m}{t} \binom{n - 1}{p^kt - 1}\\
             & =  0 \pmod{p}.
\end{align}

Since $m$ and $p^s$ are co-prime, $m = \sum_{i = 0}^j a_ip^i$
where $a_0 \geq 1$. Hence $n = p^s m = \sum_{i = s}^{s + j}
a_{i-s}p^i$, whereas $p^s = 1.p^s + \sum_{i \neq s} 0\times
p^i$. Using Lemma \ref{nondivbinom} we find that $p \nmid
\binom{n}{p^s}$.  Thus $\Delta_{fa}(x)$ has degree $p^s (m - 1)$.
\end{proof}

\begin{cor}\label{cor:degree} 
Let $a\in\mathbb{F}_{p^r}^*$.  If $\Delta_{fa}(x)\in\mathbb{F}_{p^r}[x]$ has degree $p^sl$, where $(l,p)=1$ and $0\leq s\leq r$, then \[f(x)=g(x)+\sum_{t=0}^{s}b_tx^{p^t(p^{s-t}l+1)}\]
 where at least one of $b_t\in\mathbb{F}_{p^r}^*$, and $g(x)$ is  such that   $\Delta_{ga}(x)$ is of degree  less than $p^sl$.
\end{cor}
\begin{proof}
From Theorem \ref{thm:m-1}, if $f$ has degree $p^tm$ then $\Delta_{fa}(x)$ has degree $p^t(m-1)$.
\begin{align}
p^t(m-1)= & p^sl\\
m= &  p^{s-t}l+1.
\end{align}
Thus the possible monomials $f$ for which $\Delta_{fa}(x)$ has degree $p^sl$ are of degree $p^t(p^{s-t}l+1)$ where $0\leq t\leq s$.
\end{proof}

\begin{lem}\cite{CM97D}
Let $L(x)$ and $L'(x)$ be additive permutation polynomials, and $M(x)$ an additive polynomial on a field $\mathbb{F}$ of characteristic $p$.
 Let $\Pi'(x)=L'(\Pi(L(x)))+M(x)+\delta$.   If $\Pi$ is a planar function on a field $\mathbb{F}$, then $\Pi'$ is also a planar function on $\mathbb{F}$ for all $\delta\in \mathbb{F}$.
\end{lem}
The functions $\Pi$ and $\Pi'$ are considered equivalent \cite{CD04}. 
For a field of characteristic $p$, an additive polynomial has the shape 
\begin{equation}
M(x)=\sum_{i=0}^{k}a_ix^{p^i}.
\end{equation}
The families of planar functions are specified by conditions on the degree of the monomials which make up $\Pi$.  Hence we are only considering $L(x), L'(x)$ to have degree $1$. 
Any polynomial on $\mathbb{F}_q$ may be reduced modulo $x^q-x$ to yield a polynomial of degree less than $q$ which induces the same function on $\mathbb{F}_q$ \cite{CM97D}.  Hence we only consider polynomials of degree less that $q$.

With the aid of 
the preceding facts about polynomial expansions, we show that 
 no Alltop type functions exist in fields of characteristic $3$, and that a specific class of planar functions has a unique Alltop type function.

A more recent and extensive list of planar function can  be
found in \cite{PZ2010}.  New planar functions are
continually being discovered.  The results presented here are  not an exhaustive investigation, but show some promising directions for future work.

\section{Specific classes of planar functions}
It is known that planar functions do not exist in field of characteristic $2$ \cite{dem97}.  We show that Alltop type functions cannot exist in a field of characteristic $3$.
\begin{thm}\label{thm:char3}
There are no Alltop type polynomials over $\mathbb{F}_{3^r}$.
\end{thm}
\begin{proof}
\begin{align}
\Delta_{fa}(x)= & f(x+a)-f(x)\\
\Delta\Delta_{fab}(x)= & f(x+a+b)-f(x+b)-f(x+a)\nonumber\\
 & +f(x).
\end{align}
In a field of characteristic $3$, $-1\equiv 2$, hence
\begin{align}
\Delta\Delta_{fab}(x)= & f(x+a+b)+2f(x+b)\nonumber\\
 & +2f(x+a)+f(x).
\end{align}
Let $a=b=1$ then
\begin{align}
\Delta\Delta_{f11}(x)= & f(x+2)+2f(x+1)+2f(x+1)\nonumber\\
& +f(x)\\
 =&  f(x+2)+f(x+1)+f(x).
\end{align}
Now  let $x=1,0$
\begin{align}
\Delta\Delta_{f11}(0)= & f(2)+f(1)+f(0),\\
\Delta\Delta_{f11}(1)= & f(0)+f(2)+f(1),\\
 = & \Delta\Delta_{f11}(0).
\end{align}
Hence $\Delta\Delta_{f11}(x)$ is not a permutation polynomial, $\Delta_{f1}(x)$ is not a planar function, and $f$ is not an Alltop type polynomial.
\end{proof}

This is of particular importance since while new planar functions are continually being discovered, half  of the known classes of planar functions exist only  on fields of characteristic $3$ \cite{PZ2010}.  On the other hand, the relative known abundance of planar functions on fields of  characteristic $3$, may be more a product of the ease of search, than the rarity of planar functions on fields of higher characteristic.

\begin{thm}\cite{DO68}
Let $\Pi_1(x)=x^{p^k+1}$ on $\mathbb{F}_{p^r}$,
where $p$ is an odd prime, $k\geq 0$ is an integer and
$\frac{r}{gcd(r,k)}$ is an odd integer.  Then $\Pi_1(x)$ is a planar function.
\end{thm}
This includes $x^2$ is a special case.  We now show that a cubic is the only Alltop type polynomial for this class of planar functions, conditional that $p=5$.


\begin{thm}\label{thmclass1}
Let $\Pi_1(x)=x^{p^k+1}$ on $\mathbb{F}_{p^r}$, where $k\geq 0
$ is an integer and $\frac{r}{gcd(r,k)}$ is an odd integer.
If  for each $a\in \mathbb{F}_{p^r}^*$ there exist $\alpha_a,\beta_a,\delta_a\in
\mathbb{F}_{p^r}$,  an additive polynomial $M_a(x)$ and  a polynomial $f_a(x)$  such that,
\begin{eqnarray}\label{eqn:Mdelta}
\Delta_{f_aa}(x)=\alpha_a\Pi_1(x+\beta_a)+M_a(x)+\delta_a\end{eqnarray} then $p \geq 5$ and $f_a(x)$ is equivalent to a polynomial of degree 3.
\end{thm}
\begin{proof}
Theorem \ref{thm:char3} shows that $p\geq 5$.
The proof proceeds by establishing a set of possible
degrees for $f$, and eliminating all possibilities other than $3$.  For ease of notation let $\alpha_a=\alpha$, $\beta_a=\beta$, $\delta_a=\delta$, $M_a(x)=M(x)$, and $f_a(x)=f(x)$.

It is assumed that $\Delta_{fa}(x)=\alpha\Pi_1(x+\beta)+M(x)+\delta$. Since $\alpha\Pi_1(x+\beta)+M(x)+\delta$ has
 a term of degree $p^k+1$,  and hence by Corollary \ref{cor:degree}
$f$ has a term of degree $p^k+2$.

Consider a general polynomial function $f(x)$
of degree $n$; $f(x)$ takes the form $f(x)=\sum_{i=0}^n
a_ix^i$, with $a_i\in \mathbb{F}_{p^r}$. That is, $f$ can
be written as the sum of monomials $f_i(x)=a_ix^i$ of
degree $i$ or equivalently $f(x)=\sum_{i=0}^n f_i(x)$.
Hence
\begin{eqnarray*}
f(x+a)-f(x)&= &  \sum_{i=0}^n f_i(x+a)- f_i(x)\\
\Delta_{fa}(x) &= & \sum_{i=1}^n \Delta_{f_i,a}(x).
\end{eqnarray*}

It follows that the degree of $\Delta_{f_i,a}(x)$ is less than or
equal to $p^k+1$ for all $i$. Each of these monomials can
be treated separately, with an argument similar to that
presented below.

 Let $f_n=x^n$   where $n=p^sm$ and $gcd(p,m)=1$, and
 $\Delta_{f_n,a}(x)$ is of degree $p^k+1$. By Theorem \ref{thm:m-1}, the degree of $\Delta_{f_n,a}(x)$ is
$p^s(m-1)$, and we know the degree of $\Pi_1$ is $p^k+1$, so we require
\begin{eqnarray*}
p^s(m-1)=p^k+1.
\end{eqnarray*}
There are three cases to consider, $k,s\geq 1$, $k=0$
or $s=0$.

\textbf{Case 1}: If $k,s\geq 1$, $p | p^s(m-1)$ but $p\not| (p^k+1)$ leading
to a contradiction.

\textbf{Case 2}: If $k=0$, then $p^s(m-1)=2$ and so $s=0$ and
$m=3$, which implies $p \geq 5$ as already shown. This is the Alltop function.

\textbf{Case 3}: If $s=0$, we assume $k\geq 1$ and search for solutions for
$n$ when for some $i$ $f_i(x)$ has degree $p^k+2$  thus
\begin{equation}\label{eqn:g}
f(x)=x^{p^k+2}+g(x),
\end{equation}
with $g(x)$ a
polynomial function such that $\Delta_{ga}(x)$ is of degree $p^k$ or less.
By assumption
 \begin{eqnarray*} \Delta_{fa}(x)&= &
\alpha(x+\beta)^{p^k+1}+M(x)+\delta. \end{eqnarray*}
Using Corollary  \ref{lem}, this can be simplified  to
\begin{align}
\Delta_{fa}(x)= & \alpha x^{p^k+1}+\alpha\beta
x^{p^k}\nonumber +\alpha\beta^{p^k}x\nonumber\\
& +\alpha\beta^{p^k+1}+\sum_{i=0}^{k}a_ix^{p^i} +\delta.\label{DeltaEqualPi}
\end{align}

On the other hand, using equation \ref{eqn:g} and Corollary \ref{lem} we get
\begin{align}
\Delta_{f a}(x)=& (x+a)^{p^k+2}-x^{p^k+2}+\Delta_{ga}(x)\nonumber\\
=& 2ax^{p^k+1}+a^2x^{p^k}+a^{p^k}x^2\nonumber\\
& +2a^{p^k+1}x
  +a^{p^k+2}
 +\Delta_{ga}(x)\label{ExpandDeltaf}
\end{align}

By comparing the coefficient of the $x^{p^k+1}$ terms in
equations \ref{DeltaEqualPi} and \ref{ExpandDeltaf} we find
that
\begin{align}
\alpha = & 2a.\label{alpha1}
\end{align}

If $\Delta_{ga}(x)$ has degree $p^k$, then by Corollary \ref{cor:degree}, 
$g(x)=\sum_{r=0}^{k}b_rx^{p^r(p^{k-r}+1)}+h(x)$  where $h(x)$ is a polynomial such that $\Delta_{ha}(x)$ is a polynomial of degree $p^{k}-1$ or less. Let $b'=\sum_{r=0}^k \binom{p^k+p^r}{p^r}b_ra^{p^r}$.

\begin{align}
\Delta_{f a}(x)=& (x+a)^{p^k+2}-x^{p^k+2}+\Delta_{ga}(x)\\
=& 2ax^{p^k+1}+ (a^2+
b') x^{p^k}\nonumber\\
& + \sum_{r=0}^k \binom{p^k+p^r}{p^r}b_ra^{p^k}x^{p^r} \nonumber\\
& +a^{p^k}x^2 +2a^{p^k+1}x
 +a^{p^k+2}\nonumber\\
& +\sum_{r=0}^k b_ra^{p^r(p^{k-r}+1)}+\Delta_{ha}(x)\label{ExpandDeltafg}
\end{align}
There exits  an additive polynomial $M$ in equation (\ref{eqn:Mdelta}) that can equate the  coefficients of any term of the $x^{p^i}$
terms in equations (\ref{DeltaEqualPi}) and
(\ref{ExpandDeltafg}). 

Note that equation (\ref{DeltaEqualPi}) has no $x^2$ term but equation (\ref{ExpandDeltafg}) does. Hence the coefficient of the   $x^2$ term in $\Delta_{ha}(x)$ must be
nonzero, and must cancel with the $x^2$ term already present in equation (\ref{ExpandDeltafg}).

However, we note that all the higher order terms are in
agreement, hence all such terms in $\Delta_{ha}(x)$ must have
zero coefficients implying $\Delta_{ha}(x)$ has degree 2, and
consequently  $h(x)$ is equivalent to a polynomial of degree $3$. Thus let
$h(x)=tx^3+ux^2+M(x)+w$ where $M'(x)$ is an additive polynomial.  Equation (\ref{ExpandDeltaf}) becomes

\begin{align}
\Delta_{f a}(x)= & 2ax^{p^k+1}+ (a^2+
b') x^{p^k}\nonumber\\
& + \sum_{r=0}^k \binom{p^k+p^r}{p^r}b_ra^{p^k}x^{p^r} \nonumber\\
& +(a^{p^k}+3ta)x^2 \nonumber\\
& +(2a^{p^k+1}+3ta^2+2ua)x
 +a^{p^k+2}\nonumber\\
& +\sum_{r=0}^k b_ra^{p^r(p^{k-r}+1)}+ta^3+ua^2\nonumber\\
& +\Delta_{M'a}(x)\label{eqn:longf}
\end{align}


Note that $t,u, \in \mathbb{F}_{p^r}$ and are fixed for $f(x)$.
The  coefficient of the $x^2$ term in equation
(\ref{eqn:longf}) is  $\binom{p^k+2}{2}a^{p^k}+3ta$ while in
equation (\ref{DeltaEqualPi}) the coefficient of $x^2$ is
zero. Thus
\begin{align}
0 = & a^{p^k}+3ta\nonumber\\
 = & a[a^{p^k-1}+3t]\nonumber\\
 = & a^{p^k-1}+3t\label{a}
\end{align}
In  equation (\ref{a}) $p,\ k$ and $t$ are fixed and $a$
can take any value in the field. Under the given
assumptions,
\begin{align}
a^{p^k-1}=a'^{p^k-1}
\end{align}
for all $a,a'\in\mathbb{F}^*_{p^r}$, hence
\begin{align}
a^{p^k-1}=1\label{identity}
\end{align}
 for all $a\in \mathbb{F}^*_{p^r}$.  Consequently $x^{p^k+2}\equiv x^3$, and $x^{2p^k}\equiv x^2$.  Note that equation \ref{identity} implies that $k=0$ or $r$ divides $k$.  Hence $f$ is equivalent to a polynomial of  degree $3$ which is already shown to be a valid solution in case 2.
\end{proof}

\section{Conclusion}
We have shown that for a specific family of Planar functions, a cubic is the only Alltop type polynomial.
We have also shown that Alltop type functions cannot exist on fields of characteristic $3$, which means that Alltop type functions cannot exist for many classes of planar functions.

 New planar polynomials are continually being discovered. Thus investigating the existence of Alltop type polynomials for all classes of planar function cannot yet be completed.  However many of the newly discovered planar functions are on fields of characteristic $3$. So perhaps the possible solution space is not expanding so rapidly.

 The question of the existence of another Alltop type polynomial is still open.  As is the question of whether any Alltop type polynomial would produce a  set of MUBs which are non-equivalent to the corresponding planar function MUBs.

\section*{Acknowledgment}
Thanks are due to  the referees who noted simplifications in the proof of Theorem \ref{thmclass1}.

This work was partially supported by an Australian Mathematical Society Lift-Off fellowship.



%

\end{document}